\documentclass[12pt]{amsart}
\usepackage{amscd,amssymb,amsthm,amsmath,enumerate,textcomp}
\usepackage[matrix,arrow]{xy}
\usepackage{url}
\usepackage{soul}
\usepackage{hyperref}

\theoremstyle{definition}
\newtheorem{example}[equation]{Example}
\newtheorem*{example*}{Example}

\newtheorem{theorem}[equation]{Theorem}

\newtheorem{proposition}[equation]{Proposition}
\newtheorem{conjecture}[equation]{Conjecture}
\newtheorem*{conjecture*}{Conjecture}
\newtheorem*{maintheorem*}{Main Theorem}
\newtheorem*{corollary*}{Corollary}

\newtheorem*{question*}{Question}

\newtheorem*{problem*}{Problem}
\newtheorem*{theorem*}{Theorem}

\theoremstyle{remark}
\newtheorem{remark}[equation]{Remark}
\newtheorem*{remark*}{Remark}

\makeatletter\@addtoreset{equation}{section} \makeatother

\usepackage{graphicx}

\author{Ivan Cheltsov and Victor Przyjalkowski}

\title{Fibers over infinity of Landau--Ginzburg models}

\thanks{Ivan Cheltsov was supported by the EPSRC Grant Number EP/V054597/1. The work of Victor Przyjalkowski was performed at the Steklov International Mathematical Center and supported by the Ministry of Science and Higher Education of the Russian Federation (agreement no. 075-15-2022-265).  He is a Young Russian Mathematics award winner and would like to thank its sponsors and jury.}
\address{Steklov Mathematical Institute of Russian Academy of Sciences, Moscow, Russia, \\
8 Gubkina street, Moscow 119991, Russia.}

\address{\emph{Ivan Cheltsov}
\newline
\textnormal{School of Mathematics, The University of Edinburgh,  Edinburgh, UK, \\
Edinburgh EH9 3JZ, UK.}
\newline
\textnormal{\texttt{I.Cheltsov@ed.ac.uk}}}

\address{\emph{Victor Przyjalkowski}
\newline
\textnormal{Steklov Mathematical Institute of Russian Academy of Sciences, Moscow, Russia, \\
8 Gubkina street, Moscow 119991, Russia.}
\newline
\textnormal{\texttt{victorprz@mi-ras.ru, victorprz@gmail.com}}}

\begin{document}

\begin{abstract}
We conjecture that the~number of components of the~fiber over infinity of Landau--Ginzburg model for a smooth Fano variety $X$
equals the~dimension of the~anticanonical system of $X$. We verify this conjecture for log Calabi--Yau compactifications
of toric Landau--Ginzburg models for smooth Fano threefolds, complete intersections in projective spaces, and some toric varieties.
\end{abstract}

\maketitle

\section{Introduction}
\label{section:intro}
Let $X$ be a smooth Fano variety of dimension $n$.
Then its \emph{Landau--Ginzburg model} is a certain pair $(Y,\mathsf{w})$ that consists of a smooth (quasi-projective) variety $Y$ of dimension~$n$
and a regular function
$$
\mathsf{w}\colon Y\to \mathbb{A}^1,
$$
which is called a superpotential. (More precise, Landau--Ginzburg model corresponds to a variety together with a divisor class on it,
but we assume this class to be anticanonical.)
Its fibers are compact and $K_Y\sim 0$,
so that general fiber of $\mathsf{w}$ is a smooth Calabi--Yau variety of dimension $n-1$.
Homological Mirror Symmetry \mbox{conjecture} predicts that the~derived category of singularities
of the~singular fibers of $\mathsf{w}$ is equivalent to the~Fukaya category of the~ variety $X$,
while the~Fukaya--Seidel category of the~pair $(Y,\mathsf{w})$ is equivalent
to the~bounded derived category of coherent sheaves on $X$.
In~short: the~geometry of $X$ should be determined by singular fibers of $\mathsf{w}$.

Often, Landau--Ginzburg models of smooth Fano varieties can be constructed via their toric degenerations (see \cite{Prz18}).
In this case, the~variety $Y$ contains a torus $(\mathbb{C}^*)^n$,
one has $K_Y\sim 0$, and there exists a commutative diagram
\begin{equation}
\label{equation:CCGK-compactification}
\xymatrix{
(\mathbb C^*)^n\ar@{^{(}->}[rr]\ar@{->}[d]_{\mathsf{p}}&&Y\ar@{->}[d]^{\mathsf{w}}\\
\mathbb{C}\ar@{=}[rr]&&\mathbb{C}}
\end{equation}
for some Laurent polynomial $\mathsf{p}\in\mathbb{C}[x_1^{\pm 1},\ldots,x_n^{\pm 1}]$
which is defined by an appropriate toric degeneration of the~variety $X$.
Then $\mathsf{p}$ is said to be a \emph{toric Landau--Ginzburg~model} of the~Fano variety $X$,
and $(Y,\mathsf{w})$ is said to be its \emph{Calabi--Yau compactification}.

If $(Y,\mathsf{w})$ is a Calabi--Yau compactification of a toric Landau--Ginzburg~model,
then the~number of reducible fibers of the~morphism $\mathsf{w}\colon Y\to\mathbb{C}$
does not depend on the~choice of the~Calabi--Yau compactification.
Likewise, the~number of irreducible components of each singular fiber of $\mathsf{w}$
does not depend on the~compactification either. Therefore, it is natural to expect that these numbers contain some information about the~smooth Fano~variety~$X$.
For instance, we have the~following.

\begin{conjecture}[{see \cite{Prz13, PS15, GKR17}}]
\label{conjecture:finite-fibers}
Let $X$ be a smooth Fano variety of dimension $n\ge 3$,
and let $(Y,\mathsf{w})$ be a Calabi--Yau compactification of its toric Landau--Ginzburg model.
Then
$$
h^{1,n-1}(X)=\sum_{P\in\mathbb{C}^1}\big(\rho_P-1\big),
$$
where $\rho_P$ is the~number of irreducible components of the~fiber $\mathsf{w}^{-1}(P)$.
\end{conjecture}

Note that the toric Landau--Ginzburg models considered in Conjecture~\ref{conjecture:finite-fibers}
correspond to the anticanonical divisors on Fano varieties. It may fail for other divisors.
For instance, all singular fibers of Landau--Ginzburg models of smooth del Pezzo surfaces together
with general divisors on them have at most ordinary double points as singularities,
while ones for the anticanonical divisors are very specific.
One can formulate Conjecture~\ref{conjecture:finite-fibers} replacing the Hodge number $h^{1,n-1}(X)$
by the primitive one $h^{1,n-1}_{pr}(X)$, which is equal to the usual one for $n\ge 3$ and
is less by one for $n=2$; del Pezzo surfaces satisfy the corrected Conjecture~\ref{conjecture:finite-fibers}.

This conjecture under some mild conditions can be derived from Homological Mirror Symmetry conjecture, cf.~\cite{KKP17} and~\cite{Ha17}.
Recently, Conjecture~\ref{conjecture:finite-fibers} has been verified  for Calabi--Yau compactifications
of toric Landau--Ginzburg models of smooth Fano complete intersections and smooth Fano threefolds (see~\cite{Prz13,PS15,CP18}).

In all considered cases, the~commutative diagram \eqref{equation:CCGK-compactification}
can be extended to a commutative diagram
\begin{equation}
\label{equation:KKP}
\xymatrix{
(\mathbb C^*)^n\ar@{^{(}->}[rr]\ar@{->}[d]_{\mathsf{p}}&&Y\ar@{->}[d]^{\mathsf{w}}\ar@{^{(}->}[rr]&&Z\ar@{->}[d]^{\mathsf{f}}\\
\mathbb{C}\ar@{=}[rr]&&\mathbb{C}\ar@{^{(}->}[rr]&&\mathbb{P}^1}
\end{equation}
such that $Z$ is a smooth proper variety that satisfies certain natural geometric conditions,
e.g. the~fiber $\mathsf{f}^{-1}(\infty)$ is reduced, it has at most normal crossing singularities,
and
$$
\mathsf{f}^{-1}(\infty)\sim -K_Z.
$$
Then $(Z,\mathsf{f})$ is called the~\emph{log Calabi-Yau compactification} of the~toric Landau--Ginzburg model $\mathsf{p}$ (see \cite[Definition~3.6]{Prz18}).
Observe that the~number of irreducible components of the~fiber $\mathsf{f}^{-1}(\infty)$
does not depend on the~choice of the~log Calabi-Yau compactification.
Indeed, let $\mathsf{f}^\prime\colon Z^\prime\to\mathbb{P}^1$ be another such compactification.
Then $Z$ and $Z^\prime$ are smooth proper varieties such that there exists the~following commutative diagram:
$$
\xymatrix{
Z\ar@{-->}[rr]^\psi\ar@{->}[dr]_{\mathsf{f}} & & Z^\prime\ar@{->}[dl]^{\mathsf{f}}\\
& \mathbb{P}^1 &}
$$
where $\psi$ is a birational map that is an isomorphism away from $\mathsf{f}^{-1}(\infty)$ and $(\mathsf{f}^\prime)^{-1}(\infty)$.
On~the~other hand, both relative canonical divisors $K_{Z/\mathbb{P}^1}$ and $K_{Z^\prime/\mathbb{P}^1}$ are trivial, because
\begin{align*}
\mathsf{f}^{-1}(\infty)&\sim -K_Z,\\
(\mathsf{f}^\prime)^{-1}(\infty)&\sim -K_{Z^\prime}.
\end{align*}
Then $\psi$ is a composition of flops by \cite[Theorem~1]{Ka08},
so the~number of irreducible components of $\mathsf{f}^{-1}(\infty)$
is independent on the~choice of the~log Calabi-Yau compactification.
Thus, one can expect that this number keeps some information about the~Fano variety~$X$.
The following two examples confirm this.

\begin{example}
\label{example:del-Pezzo}
Let $X$ be a smooth del Pezzo surface,
and let $(Z,\mathsf{f})$ be a log Calabi--Yau compactification of its toric Landau--Ginzburg model constructed in~\cite{AKO06}.
Then the~fiber $\mathsf{f}^{-1}(\infty)$ consists of
$$
\chi\big(\mathcal{O}(-K_X)\big)-1=h^0\big(\mathcal{O}_X(-K_{X})\big)-1=K_X^2
$$
irreducible rational curves.
\end{example}

\begin{example}
\label{example:threefolds-very-ample}
Let $X$ be a smooth Fano threefold such that the~divisor $-K_X$ is very ample,
and let $(Z,\mathsf{f})$~be a log Calabi--Yau compactification of its toric Landau--Ginzburg model constructed in \cite{ACGK,Prz17,CCGK16}.
Then  $\mathsf{f}^{-1}(\infty)$ consists~of
$$
\chi\big(\mathcal{O}(-K_X)\big)-1=h^0\big(\mathcal{O}_X(-K_{X})\big)-1=\frac{(-K_X)^3}{2}+2
$$
irreducible rational surfaces by \cite[Corollary~35]{Prz17}, see also~\cite[Theorem 2.3.14]{Ha16}.
\end{example}

This example motivates the following conjecture.

\begin{conjecture}
\label{conjecture:linear systems}
Let $X$ be a smooth Fano variety, and let $(Z,\mathsf{f})$ be a log Calabi--Yau compactification of its toric Landau--Ginzburg model.
Then $\mathsf{f}^{-1}(\infty)$ consists of
$$
\chi\big(\mathcal{O}(-K_X)\big)-1=h^0\big(\mathcal{O}_X(-K_{X})\big)-1
$$
irreducible components.
\end{conjecture}

In~\cite[Conjecture 2.3.13]{Ha16} this conjecture for threefolds is formulated in the equivalent form:
the number of components of $\mathsf{f}^{-1}(\infty)$ is equal to the genus of Fano threefold $X$
(which by definition is a genus of a generic double anticanonical section of $X$) plus $1$.
This form suggests the generalization of the latter conjecture to higher dimensions.
More precise, let $Z$ be a generic double anticanonical section of the Fano variety $X$ of dimension~$n$.
Then
$$
h^0(\mathcal{O}_X(-K_X))-1=h^0(\mathcal{O}_Z(K_Z))+1 = h^{0,n-2}(Z)+1.
$$
In~\cite[Remark 2.3.16]{Ha16}
this observation is generalized to other Hodge numbers. That is, Mirror Symmetry expectation is that the
fiber $\mathsf{f}^{-1}(\infty)$ is a mirror dual object to $Z$, and Hodge diamond for $Z$ after
the mirror $90^\circ$-rotation coincide to the Hodge diamond for the sheaf of vanishing cycles
for $\mathsf{f}$ at infinity (after appropriate shift). Thus, Conjecture~\ref{conjecture:linear systems}
can be treated as a particular case of the conjecture alluded in~\cite[Remark 2.3.16]{Ha16}, cf.~\cite[Theorem 3.8]{Ha17}.

The main result of the~paper is the~following.

\begin{theorem}
\label{theorem:linsys conjecture holds}
Conjecture~\ref{conjecture:linear systems} holds for
\begin{itemize}
  \item[{\bf \S 2:}] standard rigid maximally-mutable toric Landau--Ginzburg models for smooth Fano threefolds;
  \item [{\bf \S 3:}] Givental's toric Landau--Ginzburg models for Fano complete intersections in projective spaces;
  \item [{\bf \S 4:}] Givental's toric Landau--Ginzburg models for toric varieties whose dual toric varieties admit crepant resolutions.
\end{itemize}

\end{theorem}

\begin{remark}
\label{remark:non-vanishing}
Conjecture~\ref{conjecture:linear systems} together with
the conjectural existence of toric Landau--Ginzburg models of smooth Fano varieties  \cite[Conjecture~3.9]{Prz18}
imply that
$$
h^0\big(\mathcal{O}_X(-K_{X})\big)\geqslant 2,
$$
which is only known for $\mathrm{dim}(X)\leqslant 5$ (see \cite[Theorem~1.7]{HV11} and \cite[Theorem~1.1.1]{HS19}).
Let us also note that 
that Kawamata's \cite[Conjecture~2.1]{Ka00} implies that $h^0(\mathcal{O}_X(-K_{X}))\geqslant 1$.
\end{remark}

Homological Mirror Symmetry conjecture suggests that the~monodromy around $\mathsf{f}^{-1}(\infty)$ is maximally unipotent (see \cite[\S 2.2]{KKP17}).
Thus, if the~fiber $\mathsf{f}^{-1}(\infty)$ in \eqref{equation:KKP} is a divisor with simple normal crossing singularities,
then its dual intersection complex is expected to be homeomorphic to a sphere of dimension $n-1$ (see \cite[Question~7]{KoXu16}).
This follows from \cite[Proposition~8]{KoXu16} for $n\leqslant 5$.
However, the~following example shows that we cannot always expect $\mathsf{f}^{-1}(\infty)$ to be
a divisor with simple normal crossing singularities.

\begin{example}
\label{example:X66P1112233}
Let $X$ be a smooth intersection of two general sextics in $\mathbb{P}(1,1,1,2,2,3,3)$.
Then $X$ is a smooth Fano fourfold and $-K_X=\mathcal O(1)$, so that
$$h^0(\mathcal{O}_X(-K_{X}))-1=3-1=2.
$$
A toric Landau--Ginzburg model for $X$ is the~Laurent polynomial
$$
\mathsf{p}=\frac{(x+y+1)^6(z+t+1)^6}{x^3yz^3t},
$$
see~\cite[\S 7.2.2]{Prz18}.
The change of variables
$$
x=\frac{a^2c}{b^3d},\ \  y=\frac{ac}{b^2d}-\frac{a^2c}{b^3d}-1,\ \ z=c,\ \ t=d-c-1
$$
gives us a birational map $(\mathbb{C}^*)^4\dasharrow\mathbb{C}^4$ that maps the~pencil $\mathsf{p}=\lambda$
to the~pencil of quintics in $\mathbb{C}^4$ given by
$$
d^4=\lambda (abc-a^3c-b^3d)(d-c-1),
$$
where $\lambda$ is a parameter in $\mathbb{C}\cup\{\infty\}$.
Now arguing as in \cite{CP18}, one can construct a~log Calabi--Yau compactification $(Z,\mathsf{f})$ of the~toric Landau--Ginzburg model $\mathsf{p}$.
Then $\mathsf{f}^{-1}(\infty)$ consists of two irreducible divisors intersecting by a singular plane cubic, and the~monodromy around this fiber is maximally unipotent.
All other log Calabi--Yau compactifications differ from $(Z,\mathsf{f})$ by flops,
so that their fibers over $\infty$ also consist of two irreducible divisors.
If one of them is a divisor with simple normal crossing singularities, then
its dual intersection complex must be homeomorphic to a three-dimensional sphere by \cite[Proposition~8]{KoXu16},
which is impossible for dimension reasons.
\end{example}

Nevertheless, all toric Landau--Ginzburg models we consider in this paper admit log Calabi--Yau compactifications
such their fibers over $\infty$ are divisors with simple normal crossing singularities.
For toric Landau--Ginzburg models of smooth Fano threefolds, this follows from the~construction of
the log Calabi--Yau compactifications given in~\cite{Prz17} except for the~families \textnumero 2.1 and \textnumero 10.1.
For each of these two families, the~fiber over $\infty$ does not have simple normal crossing singularities,
but one can flop the~log Calabi--Yau compactification in several curves contained
in this fiber such that the~resulting divisor has simple normal crossing singularities.

\medskip

Let us describe the~structure of this paper.
In Section~\ref{section:Fano-threefolds} we verify Conjecture~\ref{conjecture:linear systems} for smooth Fano threefolds.
In Section~\ref{section:complete-intersections} we verify Conjecture~\ref{conjecture:linear systems} for smooth Fano complete intersections
in projective spaces.
In Section~\ref{section:toric} we verify Conjecture~\ref{conjecture:linear systems} for some smooth toric Fano varieties.

\section{Fano threefolds}
\label{section:Fano-threefolds}

In this section we prove Conjecture~\ref{conjecture:linear systems} for
\emph{standard} toric Landau--Ginzburg models of smooth Fano threefolds.
More precise, by~\cite[Theorem 4.1]{CKPT21}, mutation-equivalence classes of \emph{rigid maximally-mutable Laurent polynomials}
(see~\cite{CKPT21})
whose Newton polynomials are three-dimensional reflexive polytopes correspond one-to-one to the~98 deformation
families of three-dimensional Fano manifolds with very ample anticanonical class.
Let us call them standard.
Furthermore, each of
the 105 deformation families of three-dimensional Fano manifolds has a rigid maximally-mutable Laurent polynomial mirror
(see \cite{ACGK,Prz17,CCGK16}).
Thus for the~remaining 7 deformation families of Fano varieties with not very ample anticanonical class choose those of them that are discussed in~\cite{CP18} and call them standard as well.
Let $X$ be a smooth Fano threefold.
Then the~log Calabi--Yau compactification of its toric Landau--Ginzburg model is given by \eqref{equation:KKP},
where $\mathsf{p}$ is standard.
Let us denote by $[\mathsf{f}^{-1}(\infty)]$ the~number of irreducible components of the~fiber~$\mathsf{f}^{-1}(\infty)$.
We have to show that
$$
\big[\mathsf{f}^{-1}(\infty)\big]=\frac{(-K_X)^3}{2}+2.
$$
The polynomial $\mathsf{p}$ is not uniquely determined by $X$, but
the number $[\mathsf{f}^{-1}(\infty)]$ does not change under mutation, and thus does
depend on the~choice of $\mathsf{p}$ provided $\mathsf{p}$ is standard.
In particular, for the~very ample case we may choose $\mathsf{p}$ from~\cite{fanosearch} among any mirror partners for $X$.
Note that conjecturally Theorem~\ref{theorem:linsys conjecture holds} holds for \emph{all} rigid maximally mutable Laurent toric Landau--Ginzburg
model due to \cite[Conjecture~5.1]{CKPT21}.

By Example~\ref{example:threefolds-very-ample},
we may assume that the~anticanonical divisor $-K_X$ is not very ample, so that $X$ is a smooth Fano threefold
\textnumero 1.1, \textnumero 1.11, \textnumero 2.1, \textnumero 2.2, \textnumero 2.3, \textnumero 9.1, or \textnumero 10.1.
Here we use enumeration of deformation families of smooth Fano threefolds from~\cite{IP99}.
Recall that the~threefold $X$ can be described as follows:
\begin{itemize}
\item[(\textnumero 1.1)] a smooth sextic hypersurface in $\mathbb{P}(1,1,1,1,3)$;
\item[(\textnumero 1.11)] a smooth sextic hypersurface in $\mathbb{P}(1,1,1,2,3)$;
\item[(\textnumero 2.1)] a blow up of a smooth sextic hypersurface in $\mathbb{P}(1,1,1,2,3)$ along an elliptic curve;
\item[(\textnumero 2.2)] a double cover of $\mathbb{P}^1\times\mathbb{P}^2$ ramified in a surface of bidegree $(2,4)$;
\item[(\textnumero 2.3)] a blow up of a smooth quartic hypersurface in $\mathbb{P}(1,1,1,1,2)$ along  an elliptic curve;
\item[(\textnumero 9.1)]  $X\cong\mathbb{P}^1\times\mathbf{S}_2$, where $\mathbf{S}_2$ is a smooth del Pezzo surface of degree $2$;
\item[(\textnumero 10.1)] $X\cong\mathbb{P}^1\times\mathbf{S}_1$, where $\mathbf{S}_1$ is a smooth del Pezzo surface of degree $1$.
\end{itemize}
Moreover, it follows from  \cite[\S~2.2]{CP18}, \cite[\S~2.3]{CP18},
\cite[\S~9.1]{CP18}, \cite[\S~10.1]{CP18} and the~proof of \cite[Theorem~18]{Prz13} that we can choose the~polynomial $\mathsf{p}$ in \eqref{equation:KKP} as follows:
$$
\mathsf{p}=\left\{\aligned
&\frac{(a+b+c+1)^6}{abc}\ \text{if $X$ is a Fano threefold \textnumero 1.1},\\
&\frac{(a+b+1)^6}{ab^2c}+c\ \text{if $X$ is a Fano threefold \textnumero 1.11},\\
&\frac{(a+b+1)^6(c+1)^6}{ab^2}+\frac{1}{c}\ \text{if $X$ is a Fano threefold \textnumero 2.1},\\
&\frac{(a+b+c+1)^2}{a}+\frac{(a+b+c+1)^4}{bc}\ \text{if $X$ is a Fano threefold \textnumero 2.2},\\
&\frac{(a+b+1)^4(c+1)}{abc}+c+1\ \text{if $X$ is a Fano threefold \textnumero 2.3},\\
&\frac{(a+b+1)^4}{ab}+c+\frac{1}{c}\ \text{if $X$ is a  Fano threefold \textnumero 9.1},\\
&\frac{(a+b+1)^6}{ab^2}+c+\frac{1}{c}\ \text{if $X$ is a  Fano threefold \textnumero 10.1},
\endaligned
\right.
$$
where $(a,b,c)$ are coordinates on $(\mathbb{C}^*)^3$.

\begin{proposition}
\label{proposition:1-1-2-1-2-2-2-3-9-1}
Suppose that $X$ is a Fano threefold \textnumero 1.1, \textnumero 1.11, \textnumero 2.2, \textnumero 2.3, or \textnumero 9.1.
Then $[\mathsf{f}^{-1}(\infty)]=\frac{(-K_X)^3}{2}+2$.
\end{proposition}

\begin{proof}
It follows from  \cite{Prz13}, \cite[\S~2.2]{CP18}, \cite[\S~2.3]{CP18} and
\cite[\S~9.1]{CP18} that we can choose $\mathsf{p}$ such that
there is a pencil $\mathcal{S}$ of quartic surfaces on $\mathbb{P}^3$ given~by
$$
f_4(x,y,z,t)+\lambda g_4(x,y,z,t)=0
$$
for
$$
(f_4,g_4)=\left\{\aligned
&\big(x^4,yz(xt-xy-xz-t^2)\big)\ \text{if $X$ is a Fano threefold \textnumero 1.1},\\
&\big(x^4+z^2(xt-xy-t^2),yz(xt-xy-t^2)\big),\ \text{if $X$ is a Fano threefold \textnumero 1.11},\\
&\big(xz^3-(zt-xy-yz-t^2)z^2,xy(zt-xy-yz-t^2)\big), \ \text{if $X$ is a Fano threefold \textnumero 2.2},\\
&\big(x^3y+y(y+z)(xz+xt-t^2),z(y+z)(xz+xt-t^2)\big)\ \text{if $X$ is a Fano threefold \textnumero 2.3},\\
&\big(x^3y(y^2+z^2)(xt-xz-t^2),yz(xt-xz-t^2)\big)\ \text{if $X$ is a Fano threefold \textnumero 9.1}
\endaligned
\right.
$$
(certain changes of variables can be found in~\cite[proof of Proposition 5.11]{Prz18}),
that expands \eqref{equation:KKP} to the~following commutative diagram:
\begin{equation}
\label{equation:diagram}
\xymatrix{
(\mathbb{C}^*)^3\ar@{^{(}->}[rr]\ar@{->}[d]_{\mathsf{p}}&&Y\ar@{->}[d]^{\mathsf{w}}\ar@{^{(}->}[rr]&&Z\ar@{->}[d]^{\mathsf{f}}&& V\ar@{-->}[ll]_{\chi}\ar@{->}[d]^{\mathsf{g}}\ar@{->}[rr]^{\pi}&&\mathbb{P}^3\ar@{-->}[lld]^{\phi}\\
\mathbb{C}\ar@{=}[rr]&&\mathbb{C}\ar@{^{(}->}[rr]&&\mathbb{P}^1\ar@{=}[rr]&&\mathbb{P}^1&&}
\end{equation}
where $\phi$ is a rational map given by $\mathcal{S}$, the variety $V$~is a smooth threefold,
$\pi$ is a birational morphism described in \cite{CP18}, and $\chi$ is a composition of flops.
Here \mbox{$\lambda\in\mathbb{C}\cup\{\infty\}$},
where $\lambda=\infty$ corresponds to the~fiber $\mathsf{f}^{-1}(\infty)$.
Moreover, it follows from \cite{CP18} that $\pi$ factors
through a birational morphisms $\alpha\colon U\to\mathbb{P}^3$
that is uniquely determined by the~following three properties:
\begin{enumerate}
\item the~map $\alpha^{-1}$ is regular outside of finitely many points in $X$;
\item the~proper transform of the~pencil $\mathcal{S}$ via $\alpha$, which we denote by $\widehat{\mathcal{S}}$, is contained in the~anticanonical linear system $|-K_{U}|$;
\item for every point $P\in U$, there is a surface in $\widehat{\mathcal{S}}$ that is smooth at $P$.
\end{enumerate}
We denote by $\Sigma$ the~(finite) subset in $X$ consisting of all indeterminacy points of $\alpha^{-1}$.

Let $S$~be the~quartic surface given by $g_4(x,y,z,t)=0$,
let $\widehat{S}$~be its proper transform on the~threefold $U$, and let
$$
\widehat{D}=\widehat{S}+\sum_{i=1}^{k}a_i E_i,
$$
where $E_1,\ldots,E_k$ are $\alpha$-exceptional surfaces,
and $a_1,\ldots,a_k$ are non-negative integers such that $\widehat{D}\sim-K_U$.
Then $\widehat{D}\in\widehat{\mathcal{S}}$.
Moreover, for any $\widehat{D}^\prime\in\widehat{\mathcal{S}}$ such that $\widehat{D}^\prime\ne\widehat{D}$, we have
$$
\widehat{D}\cdot\widehat{D}^\prime=\sum_{i=1}^{s}m_i\widehat{C}_i,
$$
where $\widehat{C}_1,\ldots,\widehat{C}_n$ are base curves of the~pencil $\widehat{\mathcal{S}}$, and $m_1,\ldots,m_s$ are positive numbers.
Without loss of generality, we may assume that
the base curves of the~pencil $\mathcal{S}$ are the~curves $\alpha(\widehat{C}_1),\ldots,\alpha(\widehat{C}_r)$ for some $r\leqslant n$.
Then we let $C_i=\alpha(\widehat{C}_i)$ for every $i\leqslant r$.

For every $i\in\{1,\ldots,n\}$, let $M_i=\mathrm{mult}_{\widehat{C}_i}(\widehat{D})$ and
$$
\delta_i=\left\{\aligned
&0\ \text{if}\ M_i=1,\\
&m_i-1\ \text{if}\ M_i\geqslant 2.\\
\endaligned
\right.
$$
Then it follows from \cite[(1.10.8)]{CP18} that
\begin{equation}
\label{equation:defect}
\big[\mathsf{f}^{-1}(\infty)\big]=\big[S\big]+\sum_{i=1}^{r}\delta_i+\sum_{P\in\Sigma}D_P,
\end{equation}
where $[S]$ is the~number of irreducible components of the~surface $S$,
and $D_{P}$ is the~\emph{defect} of the~point $P\in\Sigma$ that is defined as
$$
D_{P}=A_{P}+\sum_{\substack{i=r+1\\\alpha(\widehat{C}_i)=P}}^s\delta_i,
$$
where $A_{P}$ is the~total number of indices $i\in\{1,\ldots,k\}$ such that $a_i>0$ and $\alpha(\widehat{E}_i)=P$.
By~\cite[Lemma~1.12.1]{CP18}, we have $D_{P}=0$ if the~rank of the~quadratic form of the~(local) defining
equation of the~surface $S$ at the~point $P$ is at least $2$.

To proceed, we need the~following notation: for any subsets $I$, $J$, and $K$ in $\{x,y,z,t\}$,
we write $H_I$ for the~plane defined by setting the~sum of coordinates in $I$ equal to zero,
we write $L_{I,J}=H_I\cap H_J$, and we write $P_{I,J,K}=H_I \cap H_J \cap H_K$.

Suppose $X$ is a Fano threefold \textnumero 1.1.
Recall that  $f_4=x^4$ and $g_4=yz(xt-xy-xz-t^2)$,
so that the~pencil $\mathcal{S}$ is given~by
$$
x^4-\lambda yz(xt-xy-xz-t^2)=0.
$$
Observe that every surface in this pencil is invariant with respect to the~$\mathbb{Z}/2\mathbb{Z}$-action  given by $[x:y:z:t]\mapsto[x:z:y:t]$.
Moreover, the~base locus of the~pencil $\mathcal{S}$ consists of the~curves $L_{\{x\},\{y\}}$, $L_{\{x\},\{z\}}$, $L_{\{x\},\{t\}}$.
Thus, we have $r=3$ and, without loss of generality, we may assume that
\begin{align*}
C_1&=L_{\{x\},\{y\}},\\
C_2&=L_{\{x\},\{z\}},\\
C_3&=L_{\{x\},\{t\}}.
\end{align*}
Recall that $S=\{yz(xt-xy-xz-t^2)=0\}\subset\mathbb{P}^3$, so that
$$
S=H_{\{y\}}+H_{\{z\}}+\mathcal{Q},
$$
where $\mathcal{Q}$ is the irreducible quadric surface $\{xt-xy-xz-t^2=0\}\subset\mathbb{P}^3$,
which is singular at the point $P_{\{x\},\{t\},\{y,z\}}$.
Since $S$ is smooth at general points of the~lines $L_{\{x\},\{y\}}$, $L_{\{x\},\{z\}}$, $L_{\{x\},\{t\}}$,
we see that general surface in the pencil $\mathcal{S}$ has isolated singularities.
In particular, we have $M_1=1$, $M_2=1$, and $M_3=1$.
Moreover, if $S^\prime$ is another surface in the~pencil $\mathcal{S}$, then
$$
S\cdot S^\prime=4L_{\{x\},\{y\}}+4L_{\{x\},\{z\}}+8L_{\{x\},\{t\}},
$$
which means that $m_1=4$, $m_2=4$, and $m_3=8$.
Now, taking partial derivatives of the~polynomial $x^4-\lambda yz(xt-xy-xz-t^2)$, we see that
all surfaces in the~pencil $\mathcal{S}$ are singular at the~points
$P_{\{x\},\{y\},\{z\}}$, $P_{\{x\},\{y\},\{t\}}$, $P_{\{x\},\{z\},\{t\}}$, $P_{\{x\},\{t\},\{y,z\}}$,
and these four points are the~only singularities of a general surface in this pencil.
This shows that
$$
\Sigma=\Big\{P_{\{x\},\{y\},\{z\}}, P_{\{x\},\{y\},\{t\}},P_{\{x\},\{z\},\{t\}},P_{\{x\},\{t\},\{y,z\}}\Big\}.
$$
Thus, using \eqref{equation:defect}, we get
$$
\big[\mathsf{f}^{-1}(\infty)\big]=3+D_{P_{\{x\},\{y\},\{z\}}}+D_{P_{\{x\},\{y\},\{t\}}}+D_{P_{\{x\},\{z\},\{t\}}}+D_{P_{\{x\},\{t\},\{y,z\}}}.
$$

We claim that $D_{P_{\{x\},\{y\},\{z\}}}=0$, $D_{P_{\{x\},\{y\},\{t\}}}=0$, $D_{P_{\{x\},\{z\},\{t\}}}=0$, and \mbox{$D_{P_{\{x\},\{t\},\{y,z\}}}=0$}.
Indeed, observe that $P_{\{x\},\{y\},\{z\}}\not\in\mathcal{Q}$ and $P_{\{x\},\{y\},\{z\}}\in H_{\{y\}}\cap H_{\{z\}}$,
which implies that the~rank of the~quadratic form of the~defining local equation of the~surface $S$ at the~point $P_{\{x\},\{y\},\{z\}}$ is two.
Hence, we have $D_{P_{\{x\},\{y\},\{z\}}}=0$ by \cite[Lemma~1.12.1]{CP18}.
Similarly, we see that  the~rank of the~quadratic form of the~defining equation of the~surface $S$ at the~point $P_{\{x\},\{t\},\{y,z\}}$ is three, because $P_{\{x\},\{t\},\{y,z\}}\not\in H_{\{y\}}$, $P_{\{x\},\{t\},\{y,z\}}\not\in H_{\{z\}}$,
and $\mathcal{Q}$ has an isolated ordinary double singularity at the point $P_{\{x\},\{t\},\{y,z\}}$.
This gives $D_{P_{\{x\},\{t\},\{y,z\}}}=0$.
Likewise, we have $P_{\{x\},\{z\},\{t\}}\not\in H_{\{y\}}$ and $P_{\{x\},\{z\},\{t\}}\in H_{\{z\}}\cap\mathcal{Q}$,
but both surfaces $H_{\{z\}}$ and $\mathcal{Q}$ are smooth at the point $P_{\{x\},\{z\},\{t\}}$,
and they intersect each other transversally at this point.
Hence, the~rank of the~quadratic form of the~defining equation of the~surface $S$ at the~point $P_{\{x\},\{z\},\{t\}}$ is two,
which implies that $D_{P_{\{x\},\{z\},\{t\}}}=0$ by \cite[Lemma~1.12.1]{CP18}.
Finally, keeping in mind the $\mathbb{Z}/2\mathbb{Z}$-symmetry mentioned earlier, we conclude that $D_{P_{\{x\},\{y\},\{t\}}}=0$.
Thus, we have
$$
\big[\mathsf{f}^{-1}(\infty)\big]=3+D_{P_{\{x\},\{y\},\{z\}}}+D_{P_{\{x\},\{y\},\{t\}}}+D_{P_{\{x\},\{z\},\{t\}}}+D_{P_{\{x\},\{t\},\{y,z\}}}=3=\frac{(-K_X)^3}{2}+2
$$
as claimed.

Now, we suppose that $X$ is a Fano threefold \textnumero 1.11.
Recall that $f_4=x^4+z^2(xt-xy-t^2)$ and $g_4=yz(xt-xy-t^2)$,
so that $\Sigma$ consists of the~points $P_{\{x\},\{y\},\{z\}}$,  $P_{\{x\},\{y\},\{t\}}$, $P_{\{x\},\{z\},\{t\}}$,
$r=3$, $C_1=L_{\{x\},\{z\}}$, $C_2=L_{\{x\},\{t\}}$, and
$C_3$ is the~rational quartic curve given by $y=x^4+txz^2-t^2z^2=0$.
Then $M_{1}=1$, $M_{2}=1$, $M_{3}=1$, $m_{1}=4$, $m_{2}=8$ and $m_{3}=1$,
so that
$$
\big[\mathsf{f}^{-1}(\infty)\big]=3+D_{P_{\{x\},\{y\},\{z\}}}+D_{P_{\{x\},\{y\},\{t\}}}+D_{P_{\{x\},\{z\},\{t\}}}+D_{P_{\{x\},\{t\},\{y,z\}}}=3+D_{P_{\{x\},\{y\},\{t\}}}
$$
by \eqref{equation:defect} and \cite[Lemma~1.12.1]{CP18}.
To compute $D_{P_{\{x\},\{y\},\{t\}}}$, observe that (locally)  $\alpha$  is a blow up of the~point $P_{\{x\},\{y\},\{t\}}$.
Thus, we may assume that $E_1$ is mapped to $P_{\{x\},\{y\},\{t\}}$.
Then $a_1=1$, so that $A_{P_{\{x\},\{y\},\{t\}}}=1$.
Moreover, the~pencil $\widehat{\mathcal{S}}$ has a unique base curve in~$E_1$, which is a conic in $E_1\cong\mathbb{P}^2$.
We may assume that this curve is $\widehat{C}_4$. Then $M_4=2$,
which gives $D_{P_{\{x\},\{y\},\{t\}}}=m_4$.
On the~other hand, we have
$$
10=8+\mathrm{mult}_{P_{\{x\},\{y\},\{t\}}}\big(\mathcal{C}\big)=\mathrm{mult}_{P_{\{x\},\{y\},\{t\}}}\Big(4C_1+8C_2+C_3\Big)=4+2m_4,
$$
which gives $D_{P_{\{x\},\{z\},\{t\}}}=3$, so that $[\mathsf{f}^{-1}(\infty)]=6=\frac{(-K_X)^3}{2}+2$.

Suppose that $X$ is a Fano threefold \textnumero 2.2. Recall that
$f_4=xz^3-(zt-xy-yz-t^2)z^2$ and $g_4=xy(zt-xy-yz-t^2)$,
so that
the set $\Sigma$ consists of the~points $P_{\{x\},\{y\},\{z\}}$,  $P_{\{x\},\{z\},\{t\}}$,  $P_{\{y\},\{z\},\{t\}}$,
$r=5$, $C_1=L_{\{x\},\{z\}}$, $C_2=L_{\{y\},\{z\}}$, and $C_3$, $C_4$, and $C_5$
are the~conics given by $x=zt-yz-t^2=0$, $y=xz-zt+t^2=0$, and $z=xy-t^2=0$, respectively.
Then  $M_{1}=1$, $M_{2}=1$, $M_{3}=2$, $M_{4}=1$, $M_{5}=1$, $m_{1}=2$, $m_{2}=2$, $m_{3}=2$, $m_{4}=1$, and $m_{5}=3$,
so that
$$
\big[\mathsf{f}^{-1}(\infty)\big]=\big[S\big]+1+D_{P_{\{x\},\{y\},\{z\}}}+D_{P_{\{x\},\{z\},\{t\}}}+D_{P_{\{y\},\{z\},\{t\}}}=4+D_{P_{\{y\},\{z\},\{t\}}}
$$
by \eqref{equation:defect} and \cite[Lemma~1.12.1]{CP18}.
To compute $D_{P_{\{y\},\{z\},\{t\}}}$, observe that $A_{P_{\{y\},\{z\},\{t\}}}=0$, because $S$ has a double point at $P_{\{y\},\{z\},\{t\}}$.
Moreover, locally near the~point  $P_{\{y\},\{z\},\{t\}}$, the~pencil $\mathcal{S}$ is given by
$$
\lambda y^2+z^3+z^3t-yz^2-\lambda yzt+\lambda y^2z+\lambda yt^2-yz^3-z^2t^2=0,
$$
where $P_{\{y\},\{z\},\{t\}}=(0,0,0)$.
Let $\alpha_1\colon U_1\to\mathbb{P}^3$ be the~blow up of the~point $P_{\{y\},\{z\},\{t\}}$,
and let $S^1$ be the~proper transform on  $U_1$ of the~surface $S$,
and let $\mathcal{S}^1$ be the~proper transform on  $U_1$ of the~pencil $\mathcal{S}$.
A chart of the~blow up $\alpha_1$ is given by the~ coordinate change $y_1=\frac{y}{t}$, $z_1=\frac{z}{t}$, $t_1=t$.
In~this chart, the surface  $S^1$ is given by
$$
y_1(t_1+y_1-t_1z_1+t_1y_1z_1)=0,
$$
and the~pencil $\mathcal{S}^1$~is~given~by
$$
\lambda y_1(t_1+y_1)-\lambda t_1y_1z_1+\big(\lambda t_1y_1^2z_1-t_1^2z_1^2-t_1y_1z_1^2+t_1z_1^3\big)+t_1^2z_1^3-t_1^2y_1z_1^3=0,
$$
so that all surfaces in this pencil are singular at the point $(y_1,z_1,t_1)=(0,0,0)$,
and this is the only singular point of a general surface in the~pencil $\mathcal{S}^1$ that is contained in the $\alpha_1$-exceptional surface.
Note also that the $\alpha_1$-exceptional surface contains unique base curve of the~pencil $\mathcal{S}^1$.
Without loss of generality, we may assume that its proper transform on $U$ is the~curve $\widehat{C}_6$.
Then $M_{6}=2$.
Furthermore, since the~rank of the~quadratic form of the~(local) defining
equation of the~surface $S^1$ at the~point $(y_1,z_1,t_1)=(0,0,0)$ is two,
we can apply arguments of the~proof of \cite[Lemma~1.12.1]{CP18} to the~pencil $\mathcal{S}^1$
to deduce the~equality $D_{P_{\{x\},\{z\},\{t\}}}=\delta_6=m_6-1$.
One the other hand, we have
$$
4+m_6=\mathrm{mult}_{{P_{\{y\},\{z\},\{t\}}}}\Big(3C_5+2C_1+2C_2+2C_3+C_4\Big)=6,
$$
so that $m_6=2$. This gives $D_{P_{\{y\},\{z\},\{t\}}}=1$. Hence, we have $[\mathsf{f}^{-1}(\infty)]=5=\frac{(-K_X)^3}{2}+2$.

Suppose that $X$ is a Fano threefold \textnumero 2.3.
Recall that $f_4=x^3y+y(y+z)(xz+xt-t^2)$ and $g_4=z(y+z)(xz+xt-t^2)$.
In this case, the~set $\Sigma$ consists of the~points  $P_{\{x\},\{z\},\{t\}}$, $P_{\{x\},\{y\},\{z\}}$,  $P_{\{x\},\{t\},\{y,z\}}$,
$r=5$, $C_1=L_{\{x\},\{t\}}$, $C_2=L_{\{y\},\{z\}}$, $C_3=L_{\{x\},\{y,z\}}$,
the curve $C_4$ is given by $z=x^3+xyt-yt^2=0$, and $C_5$ is the~conic $y=xz+xt-t^2=0$.
Then $M_{1}=1$,  $M_{2}=2$, $M_{3}=1$, $M_{4}=1$, $M_{5}=1$,
$m_{1}=6$, $m_{2}=2$, $m_{3}=3$, $m_{4}=1$ and $m_{5}=1$, so that
$$
\big[\mathsf{f}^{-1}(\infty)\big]=[S]+1+D_{P_{\{x\},\{z\},\{t\}}}+D_{P_{\{x\},\{y\},\{z\}}}+D_{P_{\{x\},\{t\},\{y,z\}}}=4+D_{P_{\{x\},\{z\},\{t\}}}
$$
by \eqref{equation:defect} and \cite[Lemma~1.12.1]{CP18}.
Arguing as in the~case \textnumero 1.11, we~get $D_{P_{\{x\},\{z\},\{t\}}}=2$,
so that $[\mathsf{f}^{-1}(\infty)]=6=\frac{(-K_X)^3}{2}+2$.

Finally, we consider the~case when $X$ is a smooth Fano threefold \textnumero 9.1.
In this case, we have $f_4=x^3y(y^2+z^2)(xt-xz-t^2)$ and $g_4=yz(xt-xz-t^2)$.
Then $\Sigma$ consists of the~points $P_{\{x\},\{z\},\{t\}}$ and $P_{\{x\},\{y\},\{z\}}$,
$r=4$,  $C_1=L_{\{x\},\{t\}}$, $C_1=L_{\{y\},\{z\}}$,
the curve $C_3$ is given by $y=xt-xz-t^2=0$, and $C_4$ is given by $z=x^3+yt(x+t)=0$.
Observe that $M_{1}=1$, $M_{2}=2$, $M_{3}=2$, $M_{4}=1$, $m_{1}=6$, $m_{2}=3$, $m_{3}=2$, $m_{4}=1$.
Thus, using \eqref{equation:defect} and \cite[Lemma~1.12.1]{CP18}, we get
$$
\big[\mathsf{f}^{-1}(\infty)\big]=6+D_{P_{\{x\},\{z\},\{t\}}}+D_{P_{\{x\},\{y\},\{z\}}}=6+D_{P_{\{x\},\{z\},\{t\}}}.
$$
Arguing as in the~case \textnumero 1.11, we~get $D_{P_{\{x\},\{z\},\{t\}}}=2$, so that $[\mathsf{f}^{-1}(\infty)]=\frac{(-K_X)^3}{2}+2$.
\end{proof}

\begin{proposition}
\label{proposition:2-1-10-1}
Suppose that $X$ is a Fano threefold \textnumero 2.1 or a Fano threefold \textnumero 10.1.
Then $[\mathsf{f}^{-1}(\infty)]=\frac{(-K_X)^3}{2}+2$.
\end{proposition}

\begin{proof}
It follows from \cite[\S~2.1]{CP18} that the~following commutative diagram exists:
\begin{equation}
\label{equation:diagram-P1-P2}
\xymatrix{
V\ar@{->}[drr]_{\mathsf{g}}\ar@{->}[rr]^{\pi}&&\mathbb{P}^2\times\mathbb{P}^1\ar@{-->}[d]^{\phi}&\mathbb{C}^3\ar@{_{(}->}[l]\ar@{->}[d]^{\mathsf{q}}\ar@{-->}[rr]^{\gamma}&&\mathbb{C}^\ast\times\mathbb{C}^\ast\times\mathbb{C}^\ast\ar@{^{(}->}[r]\ar@{->}[d]_{\mathsf{p}}&Y\ar@{->}[d]^{\mathsf{w}}\ar@{^{(}->}[r]&Z\ar@{->}[d]^{\mathsf{f}}\\
&&\mathbb{P}^1&\mathbb{C}^1\ar@{_{(}->}[l]\ar@{=}[rr]&&\mathbb{C}^1\ar@{=}[r]&\mathbb{C}^1\ar@{^{(}->}[r]&\mathbb{P}^1}
\end{equation}
where $\mathsf{q}$ is a surjective morphism,
$\gamma$ is a birational map that is described in \cite[\S~2.1]{CP18},
$\pi$ is a birational morphism, $V$ is a smooth threefold,
the map $\mathsf{g}$ is a surjective morphism such that
$-K_{V}\sim\mathsf{g}^{-1}(\infty)$, and $\phi$ is a rational map that is given by the~pencil $\mathcal{S}$ given by
$$
f_{2,3}(x,y,a,b,c)+\lambda g_{2,3}(x,y,a,b,c)=0,
$$
where $([x:y],[a:b:c])$ is a point in $\mathbb{P}^1\times\mathbb{P}^2$,
both $f_{2,3}$ and $g_{2,3}$ are bi-homogeneous polynomials of bi-degree $(2,3)$,
and $\lambda\in\mathbb{C}\cup\{\infty\}$.
The diagram \eqref{equation:diagram-P1-P2} is similar to~\eqref{equation:diagram},
so that we will follow the~proof of Proposition~\ref{proposition:1-1-2-1-2-2-2-3-9-1} and use its notation.
The only difference is that $\mathbb{P}^3$ is now replaced by $\mathbb{P}^1\times\mathbb{P}^2$,
and $S$ is the~surface given by $g_{2,3}(x,y,a,b,c)=0$.

Suppose that $X$ is a Fano threefold \textnumero 2.1.
Then
\begin{align*}
f_{2,3}&=x(x+y)c^3-y^2(abc-b^2c-a^3),\\
g_{2,3}&=y(x+y)(abc-b^2c-a^3).
\end{align*}
Then $\Sigma$ consists of the~point $P_{\{y\},\{a\},\{c\}}$,
and the~base locus of the~pencil $\mathcal{S}$ consists of the~curve $C_1$ given by $x+y=abc-b^2c-a^3=0$,
the curve $C_2$ given by $x=abc-b^2c-a^3=0$,
the curve $C_3$  given by $y=c=0$, and the~curve $C_4$  given by $a=c=0$.
Hence, we have
$$
\big[\mathsf{f}^{-1}(\infty)\big]=4+D_{P_{\{y\},\{a\},\{c\}}}=4=\frac{(-K_X)^3}{2}+2
$$
by \eqref{equation:defect} and \cite[Lemma~1.12.1]{CP18}, since $M_1=2$, $M_2=1$, $M_3=1$, $M_4=1$, and $m_1=2$.

Suppose that $X$ is a Fano threefold~\textnumero~10.1.
Then
\begin{align*}
f_{2,3}&=xyc^3+(x^2+y^2)(abc-b^2c-a^3),\\
g_{2,3}&=xy(abc-b^2c-a^3).
\end{align*}
In this case, we have $\Sigma=\varnothing$, and the~base locus of the~pencil $\mathcal{S}$ consists of the~curve
$C_1$ given by $x=abc-b^2c-a^3=0$, the~curve $C_2$ given by $y=abc-b^2c-a^3=0$,
and the~curve $C_3$ given by $a=c=0$.
Moreover, one has $M_1=2$, $M_2=2$, $M_3=1$, $m_1=2$ and $m_2=2$.
Hence, using \eqref{equation:defect}, we get
$$
[\mathsf{f}^{-1}(\infty)]=5=\frac{(-K_X)^3}{2}+2
$$
as claimed.
\end{proof}

\section{Fano complete intersections in projective spaces}
\label{section:complete-intersections}

Let $X$ be a Fano complete intersection in $\mathbb{P}^N$ of hypersurfaces of degrees $d_1,\ldots, d_k$,
let~$i_X$ be its Fano index, and let $\mathsf{p}$ be the~Laurent polynomial
$$
\frac{\prod_{i=1}^k(x_{i,1}+\ldots+x_{i,d_i-1}+1)^{d_i}}{\prod_{i=1}^k \prod_{j=1}^{d_i-1} x_{i,j}\prod_{j=1}^{i_X-1} y_j}+y_1+\ldots+y_{i_X-1} \in \mathbb{C}\big[x_{i,j}^{\pm 1}, y_s^{\pm 1}\big],
$$
which we consider as a regular function on the~torus $(\mathbb{C}^*)^n$, where $n=\mathrm{dim}(X)$.
Let~$\Delta$ be the~Newton polytope of $\mathsf{p}$ in $\mathcal{N}=\mathbb{Z}^n$,
let $T_\Delta$ be the~toric Fano variety whose fan polytope (convex hull of generators of rays of the~fan of $T_\Delta$) is $\Delta$.
In other words, cones of the~fan that defines $T_\Delta$ are cones of faces of $\Delta$.
Let
$$
\nabla=\Big\{x\ \big|\ \langle x,y\rangle \geqslant -1 \mbox{ for all } y\in \Delta\Big\}\subset \mathcal M_\mathbb{R}=\mathcal N^\vee\otimes \mathbb{R}
$$
be the~dual to $\Delta$ polytope.
Then $\nabla$ and $\Delta$ are \emph{reflexive} (see \cite{P16}).
Let $M$ be the~matrix
$$
\left(%
\begin{array}{rrrr|r|rrrr|rrr}
  i_X & 0 & \ldots & 0 & \ldots & 0 & 0 & \ldots & 0 & -1 & \ldots & -1 \\
  0 & i_X & \ldots & 0 & \ldots & 0 & 0 & \ldots & 0 & -1 & \ldots & -1 \\
  \ldots & \ldots & \ldots & \ldots & \ldots & \ldots & \ldots & \ldots & \ldots & \ldots & \ldots & \ldots\\
  0 & 0 & \ldots & i_X & \ldots & 0 & 0 & \ldots & 0 & -1 & \ldots & -1 \\
  -i_X & -i_X & \ldots & -i_X & \ldots & 0 & 0 & \ldots & 0 & -1 & \ldots & -1 \\
  \hline
  \ldots & \ldots & \ldots & \ldots & \ldots & \ldots & \ldots & \ldots & \ldots & \ldots & \ldots & \ldots\\
  \hline
  0 & 0 & \ldots & 0 & \ldots & i_X & 0 & \ldots & 0 & -1 & \ldots & -1 \\
  0 & 0 & \ldots & 0 & \ldots & 0 & i_X & \ldots & 0 & -1 & \ldots & -1 \\
  \ldots & \ldots & \ldots & \ldots & \ldots & \ldots & \ldots & \ldots & \ldots & \ldots & \ldots & \ldots\\
  0 & 0 & \ldots & 0 & \ldots & 0 & 0 & \ldots & i_X & -1 & \ldots & -1 \\
  0 & 0 & \ldots & 0 & \ldots & -i_X & -i_X & \ldots & -i_X& -1 & \ldots & -1 \\
\hline
  0 & 0 & \ldots & 0 & \ldots & 0 & 0 & \ldots & 0 & i_X-1 & \ldots & -1 \\
  \ldots & \ldots & \ldots & \ldots & \ldots & \ldots & \ldots & \ldots & \ldots & \ldots & \ldots & \ldots \\
  0 & 0 & \ldots & 0 & \ldots & 0 & 0 & \ldots & 0 & -1 & \ldots & i_X-1 \\
\end{array}%
\right),
$$
which is formed from $k$ blocks of sizes $(d_i-1)\times d_i$ and one last block of size $(i_X-1)\times(i_X-1)$.
Then it follows from \cite{P16} that the~vertices of $\nabla$  are the~rows of the~matrix $M$.
Note that there is a mistake in the~size of the~last block in~\cite{P16}.

It has been shown in \cite{ILP13,P16} that $\mathsf{p}$ is a toric Landau--Ginzburg~model of the~variety $X$
that admits a log Calabi--Yau compactification $(Z,\mathsf{f})$.

\begin{theorem}[{cf.~\cite[Problem 11]{P16}}]
\label{theorem:wci-components}
The number of irreducible components of the~fiber $\mathsf{f}^{-1}(\infty)$ equals $h^0(\mathcal{O}_X(-K_{X}))-1$.
\end{theorem}

\begin{proof}
By~\cite[Theorem 2.2]{ILP13}, the~toric variety $T_\Delta$ is a flat degeneration of $X$.
Since this degeneration is flat, one has
$$
\chi\left((\mathcal O_X(-K_X)\right)=\chi\left((\mathcal O_{T_\Delta}(-K_{T_\Delta})\right).
$$
On the~other hand, $T_\Delta$ is Fano by construction. Moreover, the~singularities of $T_\Delta$ are Kawamata log terminal by \cite[Proposition~3.7]{Ko95}.
Thus, by Kodaira vanishing (see e.g.~\mbox{\cite[Theorem~2.70]{KM98}}),
one has $h^i(\mathcal{O}_{T_\Delta}(-K_{T_\Delta}))=0$ for $i>0$. Similarly, applying Kodaira vanishing on a smooth Fano variety $X$, we
see that $h^i(\mathcal{O}_X(-K_X))=0$ for $i>0$. Therefore, we obtain
$$
h^0(\mathcal{O}_X(-K_X))=h^0(\mathcal{O}_{T_\Delta}(-K_{T_\Delta})).
$$
It is well known (see, for instance,~\cite[\S 6.3]{Da78}) the~anticanonical linear system of $T_\Delta$ can be described as the~linear system
of Laurent polynomials supported on
its dual polytope $\nabla$. Since $\nabla$ is reflexive, the~dimension $h^0(-K_{T_\Delta})-1$
of this linear system equals to the~number of integral points
on the~boundary of $\nabla$.
By~\cite[Theorem~1]{P16}, the~log Calabi--Yau compactification $(Z,\mathsf{f})$ is constructed via
a crepant toric resolution of $T_\Delta$ and a sequence of blow ups in smooth centers such that
exceptional divisors of these blow ups do not lie over $\infty$.
In particular, the~number of irreducible components of the~fiber $\mathsf{f}^{-1}(\infty)$
is equal to the~number of boundary divisors of the~crepant resolution of $T_\Delta$,
which is equal to the~number of integral points in the~boundary of $\nabla$,
since $\Delta$ is reflexive. This gives the~assertion of the~theorem.
\end{proof}

\begin{remark}
Theorem~\ref{theorem:wci-components} seems to hold in a much more general case of smooth Fano weighted complete intersections.
The problem is that the~Newton polytope $\Delta$ in this case is usually is not reflexive, so that $\nabla$
is not integral. This means that the~lattice points count in $\nabla$ is not enough for the~claim, because
the log Calabi--Yau compactification procedure (construction of the~diagram~\eqref{equation:KKP}) from~\cite{Prz17}
does not work.
However at least in some cases this procedure can be modified: one can construct the~compactification in the~face fan of the~(non-integral) polytope $\nabla$ and blow down some of the~components of the~fiber over infinity.
It turns out that the~blown down components correspond exactly to the~non-integral vertices of $\nabla$,
so that the~arguments of Theorem~\ref{theorem:wci-components} work in these cases.
For details see~\cite{Prz21}.
\end{remark}

\section{Toric Fano varieties}
\label{section:toric}

Let $X$ be a smooth toric Fano variety of dimension $n$, let $\Delta$ be its fan polytope, and let $\nabla$ be the~dual (integral) polytope,
and let $X^\vee$ be the~dual toric variety, i.e. the~Fano variety whose fan polytope is $\nabla$.
Note that $X^\vee$ can be singular. Suppose that $X^\vee$ admits a crepant (toric) resolution $\widetilde{X}^\vee\to X^\vee$.
Let $\mathsf{p}$ be the~Laurent polynomial given by the~sum of monomials corresponding to vertices of $\Delta$.
Then it follows from \cite{Prz17} that $\mathsf{p}$ defines
a toric Landau--Ginzburg~model of the~Fano variety $X$ that admits a log Calabi--Yau compactification
$(Z,\mathsf{f})$ such that the~exists the~following commutative diagram:
$$
\xymatrix{
\widetilde{X}^\vee\ar@{-->}[d]_{\phi}&&(\mathbb{C}^*)^n\ar@{_{(}->}[ll]\ar@{^{(}->}[rr]\ar@{->}[d]_{\mathsf{p}}&&Z\ar@{->}[d]^{\mathsf{f}}\\
\mathbb{P}^1&&\mathbb{C}\ar@{_{(}->}[ll]\ar@{^{(}->}[rr]&&\mathbb{P}^1}
$$
where $\phi$ is a rational map given by an anticanonical pencil $\mathcal{S}$ on the~(weak Fano) variety~$\widetilde{X}^\vee$.
Note that the~toric boundary divisor $\widetilde{X}^\vee\setminus (\mathbb{C}^*)^n$ is contained in $\mathcal{S}$.

\begin{proposition}
\label{proposition:toric}
The fiber $\mathsf{f}^{-1}(\infty)$ consists of $h^0(\mathcal{O}_X(-K_{X}))-1$ irreducible components.
\end{proposition}

\begin{proof}
Since $X$ is smooth, every irreducible toric boundary divisor of $\widetilde{X}^\vee$
is isomorphic to a projective space, and the~restriction of base locus of the~pencil $\mathcal{S}$ on this divisor is a hyperplane
that does not contain torus invariant points.
Thus, to obtain $Z$, we can blow up (consecutively) irreducible components of the~base locus of the~pencil $\mathcal{S}$,
which implies that $\mathsf{f}^{-1}(\infty)$ is the~proper transform of the~the toric divisor $\widetilde{X}^\vee\setminus (\mathbb{C}^*)^n$.
In particular, the~number of irreducible components of the~fiber $\mathsf{f}^{-1}(\infty)$  equals the~number of integral points of $\nabla$ minus one.
This number is exactly $h^0(\mathcal{O}_X(-K_{X}))-1$, which can be described as a linear system of Laurent polynomials supported by $\nabla$.
\end{proof}

\end{document}